\documentclass[11pt]{article}
\usepackage{mathrsfs}
\usepackage{amsfonts,amsmath,dsfont,amssymb,amsthm,stmaryrd,bbm,color}
\usepackage[hidelinks]{hyperref}

\newtheorem{thm}{Theorem}[section]
\newtheorem{prop}[thm]{Proposition}

\newtheorem{coro}[thm]{Corollary}
\newtheorem{asump}[thm]{Assumption}
\newtheorem{rmk}{Remark}

\newcommand{\R}{\mathbb{R}} 
\newcommand{\N}{\mathbb{N}}

\newcommand{\E}{\mathbb{E}}

\newcommand{\cX}{\mathcal{X}} 
\newcommand{\cY}{\mathcal{Y}}
  
\def\P{{\mathbb P}}

\begin{document}
\title{Large deviations for conditional guesswork}
\author{ 
Jiange Li \thanks{Research Laboratory of Electronics, Massachusetts Institute of Technology, Cambridge, MA 02139, USA. E-mail: jiange.li@mail.huji.ac.il}
}
\date{\today}
\maketitle

\begin{abstract}
The guesswork problem was originally studied by Massey to quantify the number of guesses needed to ascertain a discrete random variable. It has been shown that for a large class of random processes the rescaled logarithm of the guesswork satisfies the large deviation principle and this has been extended to the case where $k$ out $m$ sequences are guessed. The study of conditional guesswork, where guessing of a sequence is aided by the observation of another one, was initiated by Ar{\i}kan in his simple derivation of the upper bound of the cutoff rate for sequential decoding. In this note, we extend these large deviation results to the setting of conditional guesswork.
\end{abstract}

\section{Introduction}

Let $(X, Y)$ be a pair of random variables with $X$ and $Y$ taking values in a finite alphabet set $\cX$ and a countable alphabet set $\cY$, respectively. Here, $X$ is the random variable to be guessed by a series of truthfully answered questions of the form ``Is $X=x$?", while $Y$ is a correlated random variable that is directly observed. For example, in sequential decoding, one can think of $X$ as channel input and $Y$ as channel output. We call $G(X)$ a guessing function of $X$ if $G: \cX\mapsto\{1, 2, \cdots, |\cX|\}$ is a one-to-one function. A guessing function determines the order in which guesses are made; that is, $G(x)$ is number of queries needed when $X=x$. This guesswork problem, originally proposed by Massey \cite{Mas94}, arises for instance when a cryptanalyst must try out possible secret keys one at a time after narrowing the possibilities by some cryptanalysis. We call $G(X|Y)$ a guessing function of $X$ given $Y$, if $G(X|y)$ is a guessing function of $X$ for any given value $Y=y$. The study of conditional guesswork was initiated by Ar{\i}kan \cite{Ari96} in his simple derivation of the upper bound of the cutoff rate for sequential decoding. This was motivated by Jacobs-Berlekamp's observation  \cite{JB67} on the relationship between sequential decoding and guessing.

It is not hard to see that the average number of guesses $\E G(X)$ is minimized when one makes guesses of values of $X$ from the most likely to the least likely. Such a guessing function is called optimal. Massey \cite{Mas94} lower bounded $\E G(X)$ in terms of Shannon entropy $H(X)$. It is observed by Ar{\i}kan  \cite{Ari96} that $1/(1+\alpha)$-R\'enyi entropy $H_{1/(1+\alpha)}(X)$ is the appropriate metric to measure the logarithm of the $\alpha$-th moment $\E G(X)^\alpha$ for $\alpha>0$. (Actually, this result was discovered by Campbell \cite{Cam65} thirty years ago). Ar{\i}kan's bound is asymptotically sharp when one considers a long sequence of independent and identically distributed (i.i.d.) random variables. This result was subsequently extended by Malone-Sullivan \cite{MS04} to Markov processes with finite state spaces, and by Pfister-Sullivan \cite{PS04} to more general processes for $\alpha>-1$. This asymptotic behavior inspires the recent study of large deviations for guesswork by Christiansen-Duffy \cite{CD13} and for the guesswork of guessing $k$ out $m$ mutually independent sequences by Christiansen-Duffy-du Pin Calmon-M\'edard \cite{CDCM15}. Recently, Duffy-Li-M\'edard \cite{DLM18} employed the large deviation results of guesswork to give  a simple derivation of Shannon's channel coding theorem \cite{Sha48} for additive noise channels. Motivated by potential applications  in coding-decoding of concatenated codes, we extend these large deviation results to the setting of conditional guesswork.

%-----------------------Moments----------------------------------------
\section{Negative moments of guesses}

We derive bounds on negative moments of guessing functions, which complement Ar{\i}kan's \cite{Ari96} results on positive moments. Our lower bounds for negative moments resemble the upper bounds of Ar{\i}kan for positive moments, while our upper bound takes the form of Ar{\i}kan's lower bound for positive moments. Our proof for the lower bound mirrors Ar{\i}kan's approach, while our upper bounds proof uses a somewhat different technique, based on the reverse H\"{o}lder's inequality, inspired by a remark by Ar{\i}kan \cite{Ari96}.%These estimates are needed in the following section to establish the large deviation behavior of the conditional guesswork.

\begin{thm}\label{thm:upper}
Let $G(X)$ and $G(X|Y)$ be arbitrary guessing functions. For $\alpha\in(-1, 0)$, we have
$$\label{eq:upper-guess-u}
\E G(X)^\alpha\leq(1+\log|\cX|)^{-\alpha}\Big(\sum_{x}p_{X}(x)^{\frac{1}{1+\alpha}}\Big)^{1+\alpha},
$$
$$\label{eq:upper-guess-c}
\E G(X|Y)^\alpha\leq(1+\log|\cX|)^{-\alpha}\sum_{y}\Big(\sum_{x}p_{X, Y}(x, y)^{\frac{1}{1+\alpha}}\Big)^{1+\alpha},
$$
where $|\cX|$ is the cardinality of $\cX$, and $p_X(x)$ and $p_{X, Y}(x, y)$ are the probability mass functions of $X$ and $(X, Y)$, respectively.
\end{thm}

\begin{proof}
Recall the following reverse H\"{o}lder inequality: for any $p, q>0$ such that $1/p-1/q=1$, and any $a_i, b_i>0$, we have
$$
\sum_{i}a_ib_i\geq\Big(\sum_{i}a_i^p\Big)^{1/p}\Big(\sum_{i}b_i^{-q}\Big)^{-1/q}.
$$
The upper bound in unconditional case follows by taking $p=1+\alpha$ and $q=-(1+1/\alpha)$, and 
$$
a_i=(i^\alpha\P(G(X)=i))^{\frac{1}{1+\alpha}},~~b_i=i^{-\frac{\alpha}{1+\alpha}},~i=1,\cdots, |\cX|,
$$
as well noting that
$$
\sum_{i=1}^{|\cX|}i^{-1}\leq1+\log|\cX|.
$$
Then the upper bound in conditional case follows readily from the definition.
\begin{align*}
\E G(X|Y)^\alpha&= \sum_{y}p_Y(y)\E G(X|Y=y)^{\alpha}\\
&\leq (1+\log|\cX|)^{-\alpha}\sum_{y}p_Y(y)\Big(\sum_{x}p_{X|Y}(x|y)^{\frac{1}{1+\alpha}}\Big)^{1+\alpha}\\
&= (1+\log|\cX|)^{-\alpha}\sum_{y}\Big(\sum_{x}p_{X, Y}(x, y)^{\frac{1}{1+\alpha}}\Big)^{1+\alpha}.
\end{align*}
\end{proof}

\begin{thm}\label{thm:lower}
Let $G(X|Y)$ be an optimal conditional guessing function. For $\alpha\in(-1, 0)$, we have
\begin{align*}%\label{eq:lower-guess-c}
\E G(X|Y)^\alpha\geq\sum_{y}\Big(\sum_{x}p_{X, Y}(x, y)^{\frac{1}{1+\alpha}}\Big)^{1+\alpha}.
\end{align*}
\end{thm}

\begin{proof}
The statement can be proved in the same manner as that of Proposition 4 in \cite{Ari96}. We include the proof for completeness. For any optimal guessing function $G(X|Y)$, we have
\begin{align*}
G(x|y) &= \sum_{x': G(x'|y)\leq G(x|y)}1\\
&\leq \sum_{x': G(x'|y)\leq G(x|y)}(p_{X|Y}(x'|y)/p_{X|Y}(x|y))^{\frac{1}{1+\alpha}}\\
&\leq \sum_{x'}(p_{X|Y}(x'|y)/p_{X|Y}(x|y))^{\frac{1}{1+\alpha}}.
\end{align*}
Then the proof readily follows from
\begin{align*}
\E G(X|Y)^\alpha &=\sum_{x, y}p_{X, Y}(x,y)G(x|y)^\alpha\\
&\geq \sum_{x, y}p_{X, Y}(x,y)\Big(\sum_{x'}(p_{X|Y}(x'|y)/p_{X|Y}(x|y))^{\frac{1}{1+\alpha}}\Big)^{\alpha}\\
&= \sum_{y}p_Y(y)\Big(\sum_{x}p_{X|Y}(x|y)^{\frac{1}{1+\alpha}}\Big)^{1+\alpha}\\
&= \sum_{y}\Big(\sum_{x}p_{X, Y}(x, y)^{\frac{1}{1+\alpha}}\Big)^{1+\alpha}.
\end{align*}
\end{proof}

Recall that the R\'enyi entropy of order $\alpha>0$ (or simply $\alpha$-R\'enyi entropy) of $X$ is defined as
$$
H_\alpha(X)=\frac{\alpha}{1-\alpha}\log\Big(\sum_{x}p_X(x)^\alpha\Big)^{1/\alpha}.
$$
Unlike conditional Shannon entropy, there is no commonly accepted notation of conditional R\'enyi entropy. We refer to \cite{FB14} for discussions of different definitions. We follow Arimoto's notation \cite{Ari77} and define the conditional $\alpha$-R\'enyi entropy of $X$ given $Y$ as
\begin{align}\label{eq:cond-renyi}
H_\alpha(X|Y)=\frac{\alpha}{1-\alpha}\log\sum_{y}\Big(\sum_{x}p_{X, Y}(x, y)^\alpha\Big)^{1/\alpha}.
\end{align}
Rewrite our bounds in Theorem \ref{thm:upper} and Theorem \ref{thm:lower} in terms of R\'enyi entropies. We will see the following operational characterization of R\'enyi entropies. This connection was  identified in \cite{Cam65} and \cite{Ari96} for $\alpha>0$. %Extension for Markov processes with finite state spaces was given in \cite{MS04}, and for more general processes for $\alpha>-1$ in \cite{PS04}.

\begin{coro}\label{coro:scgf-renyi}
Let $X_{1, n}=(X_1, \cdots, X_n)$ and $Y_{1, n}=(Y_1, \cdots, Y_n)$ be two random sequences. Let $G(X_{1, n})$ and $G(X_{1, n}|Y_{1, n})$ be optimal guessing functions. Suppose the pairs $(X_i, Y_i)$ are jointly independent and have identical distribution.  Let $-1<\alpha<0$. Then we have
$$
\lim_{n\to\infty}\frac{1}{n}\log\E G(X_{1, n}|Y_{1, n})^\alpha=\alpha H_{1/{(1+\alpha)}}(X_1|Y_1).
$$
In particular, we have
$$
\lim_{n\to\infty}\frac{1}{n}\log\E G(X_{1, n})^\alpha=\alpha H_{1/{(1+\alpha)}}(X_1).
$$
\end{coro}

\begin{proof}
Since $(X_i, Y_i)$ are i.i.d., we have
\begin{align*}
\sum_{y_{1, n}}\Big(\sum_{x_{1, n}}p_{X_{1, n}, Y_{1, n}}(x_{1, n}, y_{1, n})^{\frac{1}{1+\alpha}}\Big)^{1+\alpha}=\Big(\sum_{y_1}\Big(\sum_{x_1}p_{X_1, Y_1}(x_1, y_1)^{\frac{1}{1+\alpha}}\Big)^{1+\alpha}\Big)^n.
\end{align*}
Then the statements readily follow from Theorem \ref{thm:upper} and Theorem \ref{thm:lower}.
\end{proof}

%-------------------------LDP----------------------------------------

\section{Large deviations for optimal guessing functions}

Let $\mathbb{X}=(X_1, \cdots, X_n, \cdots)$ and $\mathbb{Y}=(Y_1, \cdots, Y_n, \cdots)$ be a pair of random sequences. We denote by $X_{1, n}=(X_1, \cdots, X_n)$ and $Y_{1, n}=(Y_1, \cdots, Y_n)$ the truncated sequences of length $n$. The scaled cumulant generating function of the sequence $\{n^{-1}\log G(X_{1, n}|Y_{1, n})\}_{n\in \N}$ is defined as
\begin{align}\label{eq:scgf}
\Lambda(\alpha)=\lim_{n\to\infty}\frac{1}{n}\log\E e^{\alpha\log G(X_{1, n}|Y_{1, n})}=\lim_{n\to\infty}\frac{1}{n}\log\E G(X_{1, n}|Y_{1, n})^\alpha,
\end{align}
provided the limit exists. The conditional $\alpha$-R\'enyi entropy of $\mathbb{X}$ given $\mathbb{Y}$ is defined as
\begin{align}\label{eq:cond-renyi-vector}
H_{\alpha}(\mathbb{X}|\mathbb{Y})=\lim_{n\to\infty}\frac{1}{n}H_\alpha(X_{1, n}|Y_{1, n}),
\end{align}
provided the limit exists, and the definition of $H_\alpha(X_{1, n}|Y_{1, n})$ is given in \eqref{eq:cond-renyi}. As $\alpha\to1$, we have the classical conditional Shannon entropy. By taking limits, we have
$$
H_0(\mathbb{X}|\mathbb{Y})=\log|\cX|,
$$
\begin{align}\label{eq:h-infty}
H_\infty(\mathbb{X}|\mathbb{Y})=-\lim_{n\to\infty}\frac{1}{n}\log\P(G(X_{1, n}|Y_{1, n})=1),
\end{align}
whenever the limit exists.

Corollary \ref{coro:scgf-renyi} and its counterpart Proposition 5 in \cite{Ari96} shows that there is a close connection between the scaled cumulant generating function of conditional guesswork and the conditional R\'enyi entropy of the corresponding random sequences. The following regularity assumption, which trivially holds for i.i.d. sequences, is analogous to Assumption 1 in \cite{CD13}. It will be the base of our study of large deviations for optimal guessing functions.

\begin{asump}\label{asump:asump}
Suppose the scaled cumulant generating function $\Lambda(\alpha)$ exists for $\alpha>-1$, and it has a continuous derivative. Furthermore, 
$$
\Lambda(\alpha)=\alpha H_{1/{(1+\alpha)}}(\mathbb{X}|\mathbb{Y}).
$$
\end{asump}

\begin{prop}\label{prop:lambda-1}
Under the above assumption, for all $\alpha\leq-1$, we have
$$
\Lambda(\alpha)=-H_\infty(\mathbb{X}|\mathbb{Y}).
$$
\end{prop}

\begin{proof}
Notice that $\Lambda(\alpha)$ is the limit of a sequence of bounded convex functions. By Assumption \ref{asump:asump}, the limit $H_\infty(\mathbb{X}|\mathbb{Y})=\lim_{\alpha\downarrow-1}\Lambda(\alpha)$ exists. 
%It is clear that the sCGF can be rewritten as
%$$
%\Lambda(\alpha)=\lim_{n\to\infty}\frac{1}{n}\log\E G(X_1^n|Y_1^n)^\alpha.
%$$
By definition, we have
$$
\E(G(X_{1, n}|Y_{1, n}))^\alpha=\sum_{i=1}^{|\cX|^n}i^\alpha\P(G(X_{1, n}|Y_{1, n})=i).
$$
For $\alpha\leq-1$, we have
$$
\E(G(X_{1, n}|Y_{1, n}))^\alpha\geq\P(G(X_{1, n}|Y_{1, n})=1),
$$
and
$$
\E(G(X_{1, n}|Y_{1, n}))^\alpha\leq \P(G(X_{1, n}|Y_{1, n})=1)\sum_{i=1}^{|\cX|^n}i^{-1}.
$$
Then the result follows from the existence of $H_\infty(\mathbb{X}|\mathbb{Y})$ and the simple fact that
$$
\lim_{n\to\infty}\frac{1}{n}\log\sum_{i=2}^{|\cX|^n}i^{-1}=0.
$$
\end{proof}

The Legendre transform of $\Lambda(\alpha)$ is defined as
\begin{align}\label{eq:rate-c-single}
\Lambda^*(x)=\sup_{\alpha\in\R}(x\alpha-\Lambda(\alpha)).
\end{align}
Define $\gamma=\lim_{\alpha\downarrow-1}\Lambda'(\alpha)$. One can check that
\begin{align}\label{eq:rate-gamma}
\Lambda^*(x)=H_\infty(\mathbb{X}|\mathbb{Y})-x, ~x\in[0, \gamma],
\end{align}
and
$$
\Lambda^*(x)=\infty,~x>\log|\mathcal{X}|.
$$
%We will show that $\{n^{-1}\log G(X_1^n|Y_1^n)\}$ satisfies the LDP with the rate function $\Lambda^*(x)$. By the definition of LDP, we need to show that 
Recall that $x\in\R$ is called an exposed point of $\Lambda^*$ if for some $\alpha\in\R$ and all $x\neq y$,
$$
\alpha x-\Lambda^*(x)>\alpha y-\Lambda^*(y),
$$
and $\alpha$ is called an exposing hyperplane. 

\begin{thm}
Under Assumption \ref{asump:asump}, the sequence $\{n^{-1}\log G(X_{1, n}|Y_{1, n})\}_{n\in \N}$ satisfies the large deviation principle with the rate function $\Lambda^*(x)$, i.e.,
for any closed set $F\subset\R$,
\begin{align}\label{eq:upper}
\limsup_{n\to\infty}\frac{1}{n}\log\P(n^{-1}\log G(X_{1, n}|Y_{1, n})\in F)\leq-\inf_{x\in F}\Lambda^*(x),
\end{align}
and for any open set $J\subset\R$,
\begin{align}\label{eq:lower}
\liminf_{n\to\infty}\frac{1}{n}\log\P(n^{-1}\log G(X_{1, n}|Y_{1, n})\in J)\geq-\inf_{x\in J}\Lambda^*(x).
\end{align}
\end{thm}

\begin{proof}
%It is clear that $\{n^{-1}\log G(X_1^n|Y_1^n)\}$ is a bounded sequence. Therefore, it is exponentially tight on $\R$. 
It suffices to consider $F, J\subset [0, \log|\mathcal{X}|]$, since the sequence $\{n^{-1}\log G(X_{1, n}|Y_{1, n})\}_{n\in \N}$ is supported on this range. The upper bound \eqref{eq:upper} readily follows from G\"{a}rtner-Ellis' Theorem (Theorem 2.3.6 in \cite{DZ10:book}), which assumes the existence of $\Lambda(\alpha)$ and that 0 is in the interior of $\{\alpha\in\R: \Lambda(\alpha)<\infty\}$. These assumptions are satisfied by Assumption \ref{asump:asump} and Proposition \ref{prop:lambda-1}. Regarding the lower bound \eqref{eq:lower}, we split $[0, \log|\mathcal{X}|]$ into $[0, \gamma]$ and $(\gamma, \log|\mathcal{X}|]$. For any open set $J\subset(\gamma, \log|\mathcal{X}|]$, G\"{a}rtner-Ellis' Theorem says that
$$
\liminf_{n\to\infty}\frac{1}{n}\log\P(n^{-1}\log G(X_{1, n}|Y_{1, n})\in J)\geq-\inf_{x\in J\cap \mathcal{F}}\Lambda^*(x),
$$
where $\mathcal{F}$ is the set of exposed points of $\Lambda^*$. Assumption \ref{asump:asump} implies that $\Lambda^*(x)$ has at most a finite number of points in $(\gamma, \log|\mathcal{X}|]$ without exposing hyperplanes. The continuity of $\Lambda^*$ implies that
$$
\inf_{x\in J\cap \mathcal{F}}\Lambda^*(x)=\inf_{x\in J}\Lambda^*(x).
$$
Therefore, the lower bound \eqref{eq:lower} holds when $J\subset(\gamma, \log|\cX|]$. Owing to the representation \eqref{eq:rate-gamma}, $\Lambda^*$ has no exposed points in $[0, \gamma]$. We need a different argument for the case $J\subset[0, \gamma]$. Without loss of generality, we can assume that $\gamma>0$. For any $x\in J\subset[0, \gamma]$ and $\epsilon>0$ small enough, we have $B(x, \epsilon):=(x-\epsilon, x+\epsilon)\subset J$. One can verify that
$$
\liminf_{n\to\infty}\frac{1}{n}\log\P(n^{-1}\log G(X_{1,n}|Y_{1, n})\in J)\geq \liminf_{n\to\infty}\frac{1}{n}\log\P(n^{-1}\log G(X_{1, n}|Y_{1, n})\in B(x, \epsilon)).
$$
Since $x\in J$ is arbitrary and $\epsilon$ can be arbitrarily small, using the representation \eqref{eq:rate-gamma}, the lower bound \eqref{eq:lower} will hold if we can show that
\begin{align}\label{eq:ldp-lower}
\lim_{\epsilon\to0}\liminf_{n\to\infty}\frac{1}{n}\log\P(n^{-1}\log G(X_{1, n}|Y_{1, n})\in B(x, \epsilon))\geq x-H_\infty(\mathbb{X}|\mathbb{Y}).
\end{align}
The proof proceeds in two cases. \\
\textbf{Case 1}. There is some $x^*>\gamma$ such that $\Lambda^*(x^*)$ is finite. We select $\epsilon$ such that $x^*-\epsilon>\gamma$. Since lower bound \eqref{eq:lower} holds for $B(x^*, \epsilon)$, we have
\begin{align}\label{eq:lower-x^*}
\liminf_{n\to\infty}\frac{1}{n}\log\P(n^{-1}\log G(X_{1, n}|Y_{1, n})\in B(x^*, \epsilon))\geq -\inf_{y\in B(x^*, \epsilon)}\Lambda^*(y).
\end{align}
We define the set
\begin{align}\label{eq:set}
\cX^n(y_{1, n}, x^*, \epsilon)=\{x_{1, n}\in\cX^n: n^{-1}\log G(x_{1, n}|y_{1, n})\in B(x^*, \epsilon)\}.
\end{align}
One can verify that
\begin{align}\label{eq:card}
\lim_{\epsilon\to0}\lim_{n\to\infty}\frac{1}{n}\log |\cX^n(y_{1, n}, x^*, \epsilon)|=x^*.%\lim_{\epsilon\to0}\lim_{n\to\infty}\frac{1}{n}\log(e^{n(x^*+\epsilon)}-e^{n(x^*-\epsilon)}).
\end{align}
Notice that
\begin{align*}
&~~~~~\P(n^{-1}\log G(X_{1, n}|Y_{1, n})\in B(x^*, \epsilon))\\
&\leq \sum_{y_{1, n}\in\cY^n}|\cX^n(y_{1, n}, x^*, \epsilon)|p_{Y_{1, n}}(y_{1, n})\sup_{x_{1, n}\in \cX^n(y_{1, n}, x^*, \epsilon)}p_{X_{1, n}|Y_{1, n}}(x_{1, n}|y_{1, n})\\
&\leq \sum_{y_{1, n}\in\cY^n}|\cX^n(y_{1, n}, x^*, \epsilon)|p_{Y_{1, n}}(y_{1, n})\inf_{x_{1, n}\in \cX^n(y_{1, n}, x, \epsilon)}p_{X_{1, n}|Y_{1, n}}(x_{1, n}|y_{1, n}).
\end{align*}
In the second inequality, we use the monotonicity of conditional guesswork and the fact that $x+\epsilon<x^*-\epsilon$.  Notice that the cardinality $|\cX^n(y_{1, n}, x^*, \epsilon)|$ is independent of the choice of $y_{1, n}$. Combine the above upper bound with \eqref{eq:lower-x^*} and \eqref{eq:card}, we have
$$
-\Lambda^*(x^*)\leq x^*+\lim_{\epsilon\to0}\liminf_{n\to\infty}\frac{1}{n}\log\sum_{y_{1, n}\in\cY^n}p_{Y_{1, n}}(y_{1, n})\inf_{x_{1, n}\in \cX^n(y_{1, n}, x, \epsilon)}p_{X_{1, n}|Y_{1, n}}(x_{1, n}|y_{1, n}).
$$
Let $x^*\to\gamma$ in the above inequality. Using the representation \eqref{eq:rate-gamma}, we have
\begin{align}\label{eq:a}
\lim_{\epsilon\to0}\liminf_{n\to\infty}\frac{1}{n}\log\sum_{y_{1, n}\in\cY^n}p_{Y_{1, n}}(y_{1, n})\inf_{x_{1, n}\in \cX^n(y_{1, n}, x, \epsilon)}p_{X_{1, n}|Y_{1, n}}(x_{1, n}|y_{1, n})\geq -H_\infty(\mathbb{X}|\mathbb{Y}).
\end{align}
Notice that
\begin{align*}
&~~~~\P(n^{-1}\log G(X_{1, n}|Y_{1, n})\in B(x, \epsilon))\\
&\geq \sum_{y_{1, n}\in\cY^n}|\cX^n(y_{1, n}, x, \epsilon)|p_{Y_{1, n}}(y_{1, n})\inf_{x_{1, n}\in \cX^n(y_{1, n}, x, \epsilon)}p_{X_{1, n}|Y_{1, n}}(x_{1, n}|y_{1, n}),
\end{align*}
where $\cX^n(y_{1, n}, x, \epsilon)$ is defined as in \eqref{eq:set} with $x^*$ replaced by $x$. The lower bound \eqref{eq:ldp-lower} follows from \eqref{eq:a} and the facts that $|\cX^n(y_{1, n}, x, \epsilon)|$ is independent of $y_{1, n}$ and that as in \eqref{eq:card}
$$
\lim_{\epsilon\to0}\lim_{n\to\infty}\frac{1}{n}\log |\cX^n(y_{1, n}, x, \epsilon)|=x.
$$
This concludes the first case that there is some $x^*>\gamma$ such that $\Lambda^*(x^*)$ is finite. \\
\textbf{Case 2}. We have that $\Lambda^*(x)=\infty$ for all $x>\gamma$, which implies that $\Lambda(\alpha)$ is a linear function with slope $\gamma$ for $\alpha>-1$. To see this, using the definition of $\Lambda^*(x)$ in \eqref{eq:rate-c-single}, we can see that $x\alpha-\Lambda(\alpha)$ is monotonically increasing for $\alpha>-1$. Using the differentiability of $\Lambda(\alpha)$ in Assumption \ref{asump:asump}, we have $\Lambda'(\alpha)<x$ for any $x>\gamma$. Recall that $\gamma=\lim_{\alpha\downarrow-1}\Lambda'(\alpha)$. We have $\Lambda'(\alpha)=\gamma$ for $\alpha>-1$. The fact $\Lambda(0)=0$ and the representation \eqref{eq:rate-gamma} imply that $\gamma=H_\infty(\mathbb{X}|\mathbb{Y})$ and $\Lambda^*(\gamma)=0$. Since $\Lambda'(0)=H(\mathbb{X}|\mathbb{Y})$ is the only zero of $\Lambda^*(x)$, we must have $\gamma=H_\infty(\mathbb{X}|\mathbb{Y})=H(\mathbb{X}|\mathbb{Y})=H_0(\mathbb{X}|\mathbb{Y})=\log|\mathcal{X}|$. %Then $n^{-1}\log G(X_1^n|Y_1^n)$ is approximately uniform on $[0, \gamma]$.%
The proof in the second case proceeds by contradiction. Assume that the lower bound \eqref{eq:ldp-lower} does not hold. Then, there is some $x\in J$ such that
\begin{align*}
\lim_{\epsilon\to0}\liminf_{n\to\infty}\frac{1}{n}\log\P(n^{-1}\log G(X_{1, n}|Y_{1, n})\in B(x, \epsilon))< x-H_\infty(\mathbb{X}|\mathbb{Y}).
\end{align*}
For any fixed $\delta>0$ and $\epsilon>0$ small enough, we have
\begin{align*}
\liminf_{n\to\infty}\frac{1}{n}\log\P(n^{-1}\log G(X_{1, n}|Y_{1, n})\in B(x, \epsilon))< x-H_\infty(\mathbb{X}|\mathbb{Y})-\delta.
\end{align*}
For $n$ large enough, we have
\begin{align}\label{eq:1'}
\P(n^{-1}\log G(X_{1, n}|Y_{1, n})\in B(x, \epsilon))<e^{n(x-H_\infty(\mathbb{X}|\mathbb{Y})-\delta)}.
\end{align}
Since $(e^{n(x-\epsilon)}, e^{n\gamma}]$ can be covered by at most $\frac{e^{n\gamma}-e^{n(x-\epsilon)}}{e^{n(x+\epsilon)}-e^{n(x-\epsilon)}}$ translations of $(e^{n(x-\epsilon)}, e^{n(x+\epsilon)})$, the monotonicity of conditional guesswork and \eqref{eq:1'} imply that for $n$ large enough
\begin{align}\label{eq:1}
\P(n^{-1}\log G(X_{1, n}|Y_{1, n})\in (x-\epsilon, \gamma])&\leq \frac{e^{n\gamma}-e^{n(x-\epsilon)}}{e^{n(x+\epsilon)}-e^{n(x-\epsilon)}}\cdot e^{n(x-H_\infty(\mathbb{X}|\mathbb{Y})-\delta)}\leq\frac{e^{n(-\delta+\epsilon)}}{e^{2n\epsilon}-1},
\end{align}
which approaches 0 as $n\to\infty$. Using the definition of $H_\infty(\mathbb{X}|\mathbb{Y})$ given in \eqref{eq:h-infty}, we have that for any fixed $\delta'>0$ and $n$ large enough
$$
\P(G(X_{1, n}|Y_{1, n})=1)<e^{n(\delta'-H_\infty(\mathbb{X}|\mathbb{Y}))}.
$$ 
The monotonicity of conditional guesswork and the above upper bound imply that
\begin{align}\label{eq:2}
\P(n^{-1}\log G(X_{1, n}|Y_{1, n})\in [1, x-\epsilon])\leq e^{n(x-\epsilon)} e^{n(\delta'-H_\infty(\mathbb{X}|\mathbb{Y}))}.
\end{align}
Since $x<\gamma=H_\infty(\mathbb{X}|\mathbb{Y})$, we can select $\delta'>0$ small enough such that the above probability approaches 0 as $n\to\infty$. 
The upper bound \eqref{eq:1} together with the upper bound \eqref{eq:2} contradicts the fact that
$$
\P(n^{-1}\log G(X_{1, n}|Y_{1, n})\in [1, \gamma])=1.
$$
Hence, the lower bound \eqref{eq:ldp-lower} must hold.
\end{proof}

\begin{rmk}
Under Assumption \ref{asump:asump}, we have
\begin{align}\label{eq:log-expec-single}
\lim_{n\to\infty}\frac{1}{n}\log\E G(X_{1, n}|Y_{1, n})=\Lambda(1)=H_{1/2}(\mathbb{X}|\mathbb{Y}),
\end{align}
whereas 
\begin{align}\label{eq:expec-log-single}
\lim_{n\to\infty}\frac{1}{n}\E\log G(X_{1, n}|Y_{1, n})=\Lambda'(0)=H(\mathbb{X}|\mathbb{Y}),
\end{align}
which is the zero of the rate function $\Lambda^*(x)$.
\end{rmk}

%--------------------------------------------------------------

\section {Discussion of parallel guesswork}\label{sec:parallel}

The multi-user guesswork problem was studied by Christiansen-Duffy-du Pin Calmon-M\'edard \cite{CDCM15}. Suppose $m$ users independently select strings from a finite, but potentially large, list. An inquisitor who knows the selection probabilities of each user is equipped with a method that enables the testing of each (user, string) pair, one at a time, for whether that string had been selected by that user. The inquisitor wishes to identify any $k\leq m$ of the strings with the smallest number of total guesses. Therefore, a multi-user guesswork strategy is a querying order of (user, string) pairs. Unlike the single-user guesswork, there is no stochastically dominant strategy in the multi-user case if $k<m$ (Lemma 1, \cite{CDCM15}). The following round-robin strategy, constructed in  \cite{CDCM15}, satisfies the large deviation principle and asymptotically meets the bound of the multi-user guesswork. For the round-robin strategy, each guess allows up to $m$ parallel queries, that is, to query the most likely string of one user followed by the most likely string of a second user and so forth, for each user in a round-robin fashion, before moving to the second most likely string of each user. This is equivalent to that $m$ guessers work independently on $m$ parallel strings. We call such a strategy parallel guesswork and extend the large deviation results for single-user conditional guesswork to parallel conditional guesswork.

Define $[m]=\{1, \cdots, m\}$. For $i\in[m]$, let  
$X_{1, n}^{i}=(X_1^{i}, \cdots, X_n^{i})$ and $Y_{1, n}^{i}=(Y_1^{i}, \cdots, Y_n^{i})$ be $m$ pairs of random sequences of length $n$. Guesses for the $m$ sequences $X_{1, n}^{i}$ are made simultaneously, and we assume the outcome for $X_{1, n}^{i}$ only depends on $Y_{1, n}^{i}$. Let $G(X_{1, n}^{i}| Y_{1, n}^{i})$ be an optimal single-user conditional guessing strategy for $(X_{1, n}^{i}, Y_{1, n}^{i})$. For any $1\leq k\leq m$, we define 
\begin{align}\label{eq:k-min}
G_{k, m}(\{(X_{1, n}^{i}, Y_{1, n}^{i})\}_{i\in[m]})=k\text{-}\min(G(X_{1, n}^{1}| Y_{1, n}^{1}), \cdots, G(X_{1, n}^{m}| Y_{1, n}^{m})),
\end{align}
where $k\text{-}\min(v)$ gives the $k$-th smallest component of the vector $v$. The unconditional analogue of \eqref{eq:k-min} is studied in \cite{CDCM15}. We know that $\{n^{-1}\log G(X_{1, n}^{i}| Y_{1, n}^{i})\}_{n\in \N}$ satisfies the large deviation principle under the regularity Assumption \ref{asump:asump}. Suppose the $m$ pairs $(X_{1, n}^{i}, Y_{1, n}^{i})$ are independent. Then we can apply the contraction principle (Theorem 4.2.1 in \cite{DZ10:book}) to show that $\{n^{-1}\log G_{k, m}(\{(X_{1, n}^{i}, Y_{1, n}^{i})\}_{i\in[m]})\}_{n\in \N}$ also satisfies the large deviation principle. 

Let $\Lambda_i(\alpha)$ be the scaled cumulant generating function of $\{n^{-1}\log G(X_{1, n}^{i}| Y_{1, n}^{i})\}_{n\in \N}$ (see definition \eqref{eq:scgf}). We denote by $\Lambda_i^*(x)$ the Legendre transform of $\Lambda_i(\alpha)$. Let $H_\alpha(\mathbb{X}^{i}|\mathbb{Y}^{i})$ be the conditional $\alpha$-R\'enyi entropy  of $\mathbb{X}^{i}=(X_1^{i}, X_2^{i}, \cdots)$ given by $\mathbb{Y}^{i}=(Y_1^{i}, Y_2^{i}, \cdots)$ (definition \eqref{eq:cond-renyi-vector}).

\begin{thm}\label{thm:parallel}
Suppose the $m$ pairs $(X_{1, n}^i, Y_{1, n}^i)$ are jointly independent. Suppose that $\Lambda_i(\alpha)$ satisfies Assumption \ref{asump:asump}. Then $\{n^{-1}\log G_{k, m}(\{(X_{1, n}^{i}, Y_{1, n}^{i})\}_{i\in[m]})\}_{n\in \N}$ satisfies the large deviation principle with the rate function
\begin{align}\label{eq:rate-parallel}
I_{k, m} (x)=\max_{i_1, \cdots, i_m}\Big\{\Lambda_{i_1}^*(x)+\sum_{l=2}^k\delta_{i_l}(x)+\sum_{l=k+1}^m\gamma_{i_l}(x)\Big\},
\end{align}
where
$$
\delta_{i}(x)=
\begin{cases}
\Lambda_i^*(x) & \text{if}~x\leq H(\mathbb{X}^{i}|\mathbb{Y}^{i}), \\
0 & \text{otherwise},
\end{cases}
$$
and
$$
\gamma_{i}(x)=
\begin{cases}
\Lambda_i^*(x) & \text{if}~x\geq H(\mathbb{X}^{i}|\mathbb{Y}^{i}), \\
0 & \text{otherwise}.
\end{cases}
$$
The scaled cumulant generating function of $\{n^{-1}\log G_{k, m}(\{(X_{1, n}^{i}, Y_{1, n}^{i})\}_{i\in[m]})\}_{n\in \N}$ is
\begin{align}\label{eq:scgf-parallel}
\Lambda_{k, m}(\alpha)&=\lim_{n\to\infty}\frac{1}{n}\log\E e^{\alpha\log G_{k, m}(\{(X_{1, n}^{i}, Y_{1, n}^{i})\}_{i\in[m]})}\nonumber\\
&=\sup_{x\in[0, \log|\cX|]}(\alpha x-I_{k, m}(x)).
\end{align}

\end{thm}

We omit the proof of this statement since it can be proved in the same manner as Theorem 5 in \cite{CDCM15} with a slight change of notations. As observed in \cite{CDCM15}, the rate function is not necessary convex. Convexity of the rate function is
ensured if all users select strings using the same stochastic property, whereupon the result in Theorem \ref{thm:parallel} simplifies greatly.

\begin{coro}
Under assumptions in Theorem \ref{thm:parallel}, we also assume the $m$ pairs $(X_{1, n}^i, Y_{1, n}^i)$ have identical distribution. Let $\Lambda(\alpha)$ be the common scaled cumulant generating function with the Legendre transform $\Lambda^*(x)$. Define $H_\alpha(\mathbb{X}|\mathbb{Y})=H_\alpha(\mathbb{X}^{i}|\mathbb{Y}^{i})$ to be the common conditional $\alpha$-R\'enyi entropy. Then the rate function in \eqref{eq:rate-parallel} simplifies to
\begin{align}\label{eq:rate-parallel-iid}
I(k, m, x)=
\begin{cases}
k\Lambda^*(x), & x\in[0, H(\mathbb{X}|\mathbb{Y})],\\
(m-k+1)\Lambda^*(x), & x\in(H(\mathbb{X}|\mathbb{Y}), \log|\cX|].
\end{cases}
\end{align}
The scaled cumulant generating function in \eqref{eq:scgf-parallel} is
\begin{align}\label{eq:scgf-parallel-iid}
\Lambda_{k, m}(\alpha)=
\begin{cases}
k\Lambda\left(\frac{\alpha}{k}\right), & \alpha\leq0,\\
(m-k+1)\Lambda(\frac{\alpha}{m-k+1}), & \alpha>0.
\end{cases}
\end{align}
\end{coro}

\begin{rmk}
Since $H(\mathbb{X}|\mathbb{Y})$ is zero of $\Lambda^*(x)$, it is also the zero of $\Lambda^*_{k, m}(x)$. Similar to \eqref{eq:log-expec-single}, we have 
$$
\lim_{n\to\infty}\frac{1}{n}\log\E G_{k, m}(\{(X_{1, n}^{i}, Y_{1, n}^{i})\}_{i\in[m]})=\Lambda_{k, m}(1)=H_{\frac{m-k+1}{m-k+2}}(\mathbb{X}|\mathbb{Y}),
$$
where the second identity follows from \eqref{eq:scgf-parallel-iid} and Assumption \ref{asump:asump}. Analogous to \eqref{eq:expec-log-single}, we have
$$
\lim_{n\to\infty}\frac{1}{n}\E\log G_{k, m}(\{(X_{1, n}^{i}, Y_{1, n}^{i})\}_{i\in[m]})=\Lambda_{k, m}'(0)=H(\mathbb{X}|\mathbb{Y}).
$$
\end{rmk}

\section*{Acknowledgment}
We thank Muriel M\'edard and Ken R. Duffy for valuable comments.  We also thank Ahmad Beirami for pointing out reference \cite{Cam65}.

%\bibliographystyle{plain}
%\bibliography{$HOME/Dropbox/Research/ref}

\begin{thebibliography}{10}

\bibitem{Ari96}
E.~Ar{\i}kan, ``An inequality on guessing and its application to sequential decoding," \emph{IEEE Trans. Inform. Theory}, vol. 42, no. 1, pp. 99--105, January 1996.

\bibitem{Ari77}
S.~Arimoto, ``Information measures and capacity of order {$\alpha $} for discrete
  memoryless channels," pp. 41--52. \emph{Colloq. Math. Soc. J{\'a}nos Bolyai}, vol. 16, 1977.

\bibitem{Cam65}

L. L. Campbell, ``A coding theorem and Rényi's entropy", \emph {Informat. Contr.}, vol. 8, pp. 423-429, 1965.

\bibitem{CD13}
M.~M. Christiansen and K.~R. Duffy, ``Guesswork, large deviations, and {S}hannon entropy," \emph{IEEE Trans. Inform. Theory}, vol. 59, no. 2, pp. 796--802, February 2013.

\bibitem{CDCM15}
M.~M. Christiansen, K.~R. Duffy, F.~du~Pin~Calmon, and M~M{\'e}dard, ``Multi-user guesswork and brute force security,"
\emph{IEEE Trans. Inform. Theory}, vol. 61, no. 12, pp. 6876--6886, December 2015.

\bibitem{DZ10:book}
A.~Dembo and O.~Zeitouni, ``Large deviations techniques and applications," vol.~38, \emph{Stochastic Modelling and Applied Probability}, Springer-Verlag, Berlin, 2010.

\bibitem{DLM18}

K.~R. Duffy, J. Li, and M~M{\'e}dard, ``Capacity-achieving guessing random additive noise decoding (GRAND)," \emph{IEEE Trans. Inform. Theory}, to appear

\bibitem{FB14}
S.~Fehr and S.~Berens,  ``On the conditional {R}{\'e}nyi entropy," \emph{IEEE Trans. Inform. Theory}, vol. 60, no. 11, pp. 6801--6810, November 2014.


\bibitem{JB67}
I.~Jacobs and E.~Berlekamp, ``A lower bound to the distribution of computation for sequential
  decoding," \emph{IEEE Trans. Inform. Theorys}, vol. 13, no. 2, pp. 167--174, April 1967.

\bibitem{MS04}
D.~Malone and W.~G. Sullivan, ``Guesswork and entropy," \emph{IEEE Trans. Inform. Theory}, vol. 50, no. 3, pp. 525--526, March 2004.

\bibitem{Mas94}
J. L. Massey, ``Guessing and entropy," \emph{Proc. 1994 IEEE Int. Symp. on Information Theory}, page 204. Trondheim, Norway, 1994.

\bibitem{PS04}
C.~E. Pfister and W.~G. Sullivan, ``R{\'e}nyi entropy, guesswork moments, and large deviations," \emph{IEEE Trans. Inform. Theory}, vol. 50, no. 11, pp. 2794--2800, November 2004.

\bibitem{Sha48}
C.~E. Shannon, ``A mathematical theory of communication,'' \emph{Bell System Tech.
  J.}, vol.~27, pp. 379--423, 623--656, 1948.

\end{thebibliography}

\end{document}